\documentclass[11pt]{amsart}

\usepackage{amssymb,latexsym, amsmath, amsxtra,amsfonts,amsthm}
\usepackage{braket}
\usepackage{caption}
\usepackage{cite}
\usepackage{cleveref}
\usepackage[usenames, dvipsnames]{color} 
\usepackage{dsfont}
\usepackage{empheq}
\usepackage{epsfig}
\usepackage{etoolbox}
\usepackage{float}
\usepackage{mathtools}
\usepackage{mathrsfs}
\usepackage{mdframed}
\usepackage{mleftright}
\usepackage{physics}
\usepackage{stmaryrd}
\usepackage{subcaption}
\usepackage{thmtools,thm-restate}
\usepackage{tikz}
\usepackage{times}
\usepackage{url}
\usepackage{bm}
\usepackage[]{graphics}
\usepackage{bbm}
\usepackage[normalem]{ulem}
\usetikzlibrary{arrows,automata,matrix,positioning,calc,shapes,backgrounds,trees,decorations.text,shadows, decorations.pathreplacing}
\tikzset{cross/.style={cross out, draw=black, minimum size=2*(#1-\pgflinewidth), inner sep=0pt, outer sep=0pt},
cross/.default={1pt}}

\newtheorem*{remark}{Remark}
\newtheorem*{definition}{Definition}

\newtheorem*{conjecture}{Conjecture}
\newtheorem*{thm}{Theorem}
\newtheorem*{lemmum}{Lemma}
\newtheorem{theorem}{Theorem}[section]

\newtheorem{lemma}[theorem]{Lemma}

\newcommand{\eps}{\varepsilon}

\newcommand{\cS}{\mathcal{S}}
\newcommand{\cT}{\mathcal{T}}
\newcommand{\cV}{\mathcal{V}}

\DeclareMathOperator{\RE}{Re}

\DeclareMathOperator{\mom}{MoM}

\begin{document}
\title{On the moments of the moments of $\zeta(1/2+it)$}

\begin{abstract}
Taking $t$ at random, uniformly from $[0,T]$, we consider the $k$th moment, with respect to $t$, of the random variable corresponding to the $2\beta$th moment of 
$\zeta(1/2+ix)$ over the interval $x\in(t, t+1]$, where $\zeta(s)$ is the Riemann zeta function.   We call these the `moments of moments' of the Riemann zeta function, and present a conjecture for their asymptotics, when $T\to\infty$, for integer $k,\beta$.  This is motivated by comparisons with results for the moments of moments of the characteristic polynomials of random unitary matrices and is shown to follow from a conjecture for the shifted moments of $\zeta(s)$ due to Conrey, Farmer, Keating, Rubinstein, and Snaith~\cite{cfkrs2}.  Specifically, we prove that a function which, the shifted-moment conjecture of~\cite{cfkrs2} implies, is a close approximation to the moments of moments of the zeta function does satisfy the asymptotic formula that we conjecture.  We motivate as well similar conjectures for the moments of moments for other families of primitive $L$-functions.
\end{abstract}

\author{E. C. Bailey}
\address{School of Mathematics, University of Bristol, Bristol, BS8 1UG, United Kingdom}
\email{e.c.bailey@bristol.ac.uk}

\author{J. P. Keating}
\address{Mathematical Institute, University of Oxford, Oxford, OX2 6GG, United Kingdom}
\email{jon.keating@maths.ox.ac.uk}

\maketitle

%%%%%%%%%%%%%%%%%%%%%%%
%                     %
%       Intro         %
%                     %
%%%%%%%%%%%%%%%%%%%%%%%
\section{Introduction}\label{sec:intro}

\subsection{Moments of moments}\label{sec:mom}

Moments of $L$-functions play a central role in analytic number theory.  It is a long-standing conjecture for the Riemann zeta function, $\zeta(s)$, that for the $2\beta$th moment on the critical line $\RE(s)=1/2$, as $T\to\infty$

\begin{equation}\label{zeta_moment}
  M_\beta(T)\coloneqq\frac{1}{T}\int_0^T|\zeta(\tfrac{1}{2}+it)|^{2\beta}dt\sim a_\beta c_\beta \left(\log \frac{T }{2\pi}\right)^{\beta^2},
\end{equation} 
where $a_\beta$ is an arithmetic factor given in terms of an Euler product and $c_\beta$ is discussed below (see also, for example,~\cite{confar00}).  Hardy and Littlewood~\cite{harlit18} proved this asymptotic for $M_1(T)$, and Ingham~\cite{ing26} proved it for $M_2(T)$.  Their formulae, together with heuristic calculations by Conrey and Ghosh~\cite{congho92} and Conrey and Gonek~\cite{congon01} for $M_3(T)$ and $M_4(T)$ respectively, suggest that $c_\beta \cdot (\beta^2)!\in\mathbb{N}$ for $\beta\in\mathbb{N}$.  Beyond $\beta=2$ there are no rigorous asymptotic results for $M_\beta(T)$.  Ramachandra~\cite{ram80} and Heath-Brown~\cite{heabro81} have established the lower bound $M_\beta(T)\gg (\log\frac{T}{2\pi})^{\beta^2}$ for positive, rational $\beta$, and Soundararajan and Radziwi{\l}{\l}~\cite{sourad13} extended the result to all $\beta\geq 1$.  Upper bounds of the correct size are known, conditional on the Riemann hypothesis, due to arguments of Soundararajan~\cite{sou09} and Harper~\cite{har13}. 

There is now a relatively well developed understanding of the conjectural connections between the Riemann zeta function and the characteristic polynomials of random unitary matrices (cf.~\cite{keasna00a, cfkrs2}).  The notation we use for such a characteristic polynomial is
\begin{equation}
  P_N(A,\theta)=\det(I-Ae^{-i\theta}),
\end{equation}
for $A\in U(N)$, the group of $N\times N$ unitary matrices.  By comparing densities of eigenvalues and zeros, one identifies the matrix size $N$ with a height $\log\frac{T}{2\pi}$ up the critical line.  Under this analogy, the random matrix average corresponding to~\eqref{zeta_moment} is 
\begin{equation}
  \mathcal{M}_\beta(N)\coloneqq\int_{U(N)}|P_N(A,\theta)|^{2\beta}dA,
\end{equation}
where the measure $dA$ is the Haar measure on $U(N)$. $\mathcal{M}_\beta(N)$ can be computed when $\RE(\beta)>-1/2$, giving the following result.
\begin{thm}[Keating-Snaith~\cite{keasna00a}]
  For $\RE(\beta)>-1/2$, 
  \[\mathcal{M}_\beta(N)=\prod_{j=1}^N\frac{\Gamma(j)\Gamma(j+2\beta)}{(\Gamma(j+\beta))^2}.\]
\end{thm}
It follows that asymptotically, as $N\to\infty$,
\begin{equation}\label{rmt_moment}
  \mathcal{M}_\beta(N)\sim \frac{\mathcal{G}^2(1+\beta)}{\mathcal{G}(1+2\beta)}N^{\beta^2}\,
\end{equation}
where $\mathcal{G}(s)$ is the Barnes $\mathcal{G}$-function.  One sees that by identifying $N$ with $\log \frac{T}{2\pi}$, $\mathcal{M}_\beta(N)$ grows asymptotically like $M_\beta(T)$, leading to the conjecture~\cite{keasna00a} that 
\begin{equation}\label{coefficient_conj}
  c_\beta=\frac{\mathcal{G}^2(1+\beta)}{\mathcal{G}(1+2\beta)}.
\end{equation}
This matches the known values when $\beta=1$~\cite{harlit18} and $\beta=2$~\cite{ing26}, as well as the values previously conjectured for $\beta=3$~\cite{congho92} and $\beta=4$~\cite{congon01} using heuristic number-theoretic methods.  Recently, these heuristic number-theoretic methods have been extended to calculate all integer moments of the zeta function on its critical line, giving results that agree precisely with~\eqref{coefficient_conj} and with its extension to include lower order terms~\cite{cfkrs2} -- see~\cite{ck1, ck2, ck3, ck4, ck5}.  (The correctness of the formula in the function field setting follows, in the appropriate limit, from equidistribution; cf.~\cite{katsar99}.)

In \eqref{zeta_moment}, the moments are defined with respect to an average over an asymptotically long stretch of the critical line.  Recently there has been interest in the fluctuations in the values of moments defined with respect to averages over random short intervals, e.g.~intervals of constant length.  The moments of these fluctuations are therefore the {\it moments of moments}.   These moments of moments have played a central role in our understanding of the extreme values taken by the zeta function over short intervals~\cite{fyodorov12, fyodorov14, harper1, harper2}, and a precise conjecture for the local maximum is due to Fyodorov and Keating~\cite{fyodorov14}. Recent work of Najnudel~\cite{Najnudel} and Arguin et al.~\cite{abbrs} has led to a verification to leading order of this conjecture, with an almost sharp upper bound, including the predicted subleading order, due to Harper~\cite{harper2}.  At time of writing, the strongest result is due to Arguin et al.~\cite{argbourad20}, whose work settles the upper bound part of the conjecture including the predicted tail. Extreme values and moments over mesoscopic and macroscopic intervals, as well as the transition between the two, are also of interest and have been studied by Arguin et al.~\cite{argouirad}.

Precisely, we define the moments of moments here as follows.  
%%%%
% Diagram showing analogy
%%%%
%\begin{figure}
%\centering
%\begin{tikzpicture}
%\draw[help lines, color=gray!30, dashed] (-4.9,1.1) grid (-1.1,6.9);
%\draw[thick,color=gray!80, dashed] (-3,1)--(-3,7);
%\draw[thick,color=gray!80] (-3,2)--(-3,6);
%\draw[->,ultra thick] (-5,2)--(-1,2) node[right]{$\operatorname{Re}(s)$};
%\draw[->,ultra thick] (-4,1)--(-4,7) node[above]{$\operatorname{Im}(s)$};
%\node at (-4.2,6) {\scriptsize{$T$}};
%\node at (-2,1.8) {\tiny{$1$}};
%\node at (-3.1,1.76) {\tiny{$\tfrac{1}{2}$}};
%\draw[ultra thick, color=gray!80] (-3,3)--(-2.9,3) {};
%\draw[ultra thick, color=gray!80] (-3,4)--(-2.9,4) {};
%\draw[ultra thick, color=gray!80] (-3,4.5)--(-2.9,4.5) {};
%\draw[ultra thick, color=gray!80] (-3,5.5)--(-2.9,5.5) {};
%\draw [decorate,decoration={brace,amplitude=3pt,mirror,raise=2pt},yshift=0pt, color=gray!80]
%(-2.9,3) -- (-2.9,4) node [black,midway,xshift=0.4cm] {\tiny{\color{gray!80}$2\pi$}};
%\draw [decorate,decoration={brace,amplitude=3pt,mirror,raise=2pt},yshift=0pt, color=gray!80]
%(-2.9,4.5) -- (-2.9,5.5) node [black,midway,xshift=0.4cm] {\tiny{\color{gray!80}$2\pi$}};
%\draw (2,3.5) circle (.5cm);
%\draw (2.5,3.5) node[cross] {}; %Would add more eigenvalues!
%\draw (2,5) circle (.5cm);
%\draw (2.5,5) node[cross] {};
%\draw[ultra thick, ->, gray!80] (-2,5) -- (1,5);
%\draw[ultra thick, ->, gray!80] (-2,3.5) -- (1,3.5);
%\end{tikzpicture}
%\caption{Comparison of moments of moments of $\zeta(s)$ and $P_N(A,\theta)$}\label{fig1}
%\end{figure}
\begin{definition}
  For $T>0$ and $\RE(\beta)>-1/2$
  \begin{align}\label{def:mom_zeta}
    \mom_{\zeta_T}(k,\beta)\coloneqq \frac{1}{T}\int_0^T\left(\int_{t}^{t+1}|\zeta(\tfrac{1}{2}+ih)|^{2\beta}dh\right)^kdt.
  \end{align}
\end{definition}
Choosing to average over intervals of length $1$ is simply for notational convenience\footnote{See also~\cite{fyodorov12, fyodorov14} where the moments of moments were introduced; there the intervals are instead of length $2\pi$.}; the results stated in section~\ref{sec:results} hold for any interval that is $O(1)$ as $T\rightarrow\infty$. 

The corresponding `moments of moments' for the characteristic polynomials of random unitary matrices may similarly be defined by  
\begin{equation}\label{unitary_mom}
  \mom_{U(N)}(k,\beta)\coloneqq\int_{U(N)}\left(\frac{1}{2\pi}\int_0^{2\pi}|P_N(A,\theta)|^{2\beta}d\theta\right)^kdA.
\end{equation} 
These have recently been computed for integer values of $k,\beta$ in~\cite{baikea19, ak19}.  This suggests a conjecture for $\mom_{\zeta_T}(k,\beta)$, which we state in section~\ref{sec:results}.  We show that this conjecture follows from another conjecture due to Conrey et al.~\cite{cfkrs2} concerning the shifted moments of the zeta function.  Our proof of this follows similar lines to the corresponding proof in~\cite{baikea19}, which is based on an expression derived in~\cite{cfkrs1} for the shifted moments of characteristic polynomials of random matrices.  Moreover, by appealing to results of Assiotis, Bailey, and Keating~\cite{abk19} for the other classical compact matrix groups, in section~\ref{sec:l-function_conj} we give conjectures of moments of moments for other families of $L$-functions.

\subsection{Statement of Result}\label{sec:results}

In order to develop a conjecture for the asymptotics of $\mom_{\zeta_T}(k,\beta)$ for large $T$, we use the large $N$ behaviour of $\mom_{U(N)}(k,\beta)$. Fyodorov, Hiary and Keating~\cite{fyodorov12}, and Fyodorov and Keating~\cite{fyodorov14} conjectured the asymptotic form\footnote{Here and throughout, we write $A(t)\sim B(t)$ as $t\rightarrow\infty$ for $A(t)/B(t)\rightarrow 1$ as $t\rightarrow\infty$.} of $\mom_{U(N)}(k,\beta)$ as matrix size $N\to\infty$, 
\begin{equation}\label{conj:FK}
  \mom_{U(N)}(k,\beta)\sim
  \begin{cases}
    \left(\frac{\mathcal{G}^2(1+\beta)}{\mathcal{G}(1+2\beta)\Gamma(1-\beta^2)}\right)^k
    \Gamma(1-k\beta^2)N^{k\beta^2},
    &\text{if }k<1/\beta^2,
    \\
    c_{k, \beta}N^{k^2\beta^2-k+1},
    &\text{if }k>1/\beta^2,
  \end{cases}
\end{equation}
where again $\mathcal{G}(s)$ is the Barnes $\mathcal{G}$-function and $c_{k,\beta}$ is an unspecified function of $k$ and $\beta$.  

For integer $k,\beta$, Bailey and Keating~\cite{baikea19} determined the asymptotic growth of $\mom_{U(N)}(k,\beta)$: 
\begin{thm}[Bailey-Keating~\cite{baikea19}]
  For $k,\beta\in\mathbb{N}$, 
  \begin{equation}\label{eq:baikea}
    \mom_{U(N)}(k,\beta)=c_{k,\beta}N^{k^2\beta^2-k+1}(1+O\left(\tfrac{1}{N}\right)),
  \end{equation}
  where the leading order coefficient $c_{k,\beta}$ is explicitly given in terms of an integral over a certain simplex.  Furthermore, $\mom_{U(N)}(k,\beta)$ is a polynomial in $N$. 
\end{thm}

Via an alternative representation (involving Toeplitz determinants and a Riemann-Hilbert analysis) Claeys and Krasovsky~\cite{CK} calculated $\mom_{U(N)}(2,\beta)$ for $\beta>-1/4$, and determined an alternative expression for the coefficient $c_{2,\beta}$.  This calculation was recently extended by Fahs~\cite{fahs} to all $k\in\mathbb{N}$ and non-negative, real $\beta$, though without an explicit expression for $c_{k,\beta}$.  For integer values of $k$ and $\beta$, Assiotis and Keating showed how one of the approaches developed in~\cite{baikea19} leads to a connection with lattice point counts and gave yet another expression for $c_{k,\beta}$, as a volume of a certain region involving continuous Gelfand-Tsetlin patterns with constraints.  In a recent paper~\cite{baikea21}, the present authors computed the analogue of \eqref{unitary_mom} for a different logarithmically correlated process - branching random walks - with $k\in\mathbb{N}$ and $\beta\in\mathbb{R}$.  The resulting expression agrees asymptotically with \eqref{eq:baikea}.

By making the natural exchanges in \eqref{eq:baikea}, we therefore expect the following to hold.  
\begin{conjecture}
  Let $k,\beta\in\mathbb{N}$.  Then,
  \begin{equation}\label{conj:main}
    \mom_{\zeta_T}(k,\beta)=\alpha_{k,\beta}c_{k,\beta}\left(\log\tfrac{T}{2\pi}\right)^{k^2\beta^2-k+1}\left(1+O_{k,\beta}\left(\log^{-1} T\right)\right),
  \end{equation}
  where $c_{k,\beta}$ is the same coefficient appearing in \eqref{eq:baikea}, and $\alpha_{k,\beta}$ contains the arithmetic information and takes the form of an Euler product over primes (cf. \eqref{zeta_moment}).  
\end{conjecture}
A similar conjecture, as well as further motivation for studying moments of moments, can be found in Fyodorov and Keating~\cite{fyodorov14}.  

Furthermore, we expect that the moments of moments, when suitably smoothed, will have the general form $\operatorname{Poly}_{k^2\beta^2-k+1}\left(\log\tfrac{T}{2\pi}\right)+O_{k,\beta}\left(T^{-\delta}\right)$, for some $\delta>0$, where $\operatorname{Poly}_n(x)$ is a polynomial of degree $n$ in the variable $x$.
  
For integer $k$ we can rewrite $\mom_{\zeta_T}(k,\beta)$ by expanding and switching the order of integration,
\begin{equation}\label{eq:mom_expanded}
  \mom_{\zeta_T}(k,\beta)=\frac{1}{T}\int_{0}^{1}\cdots\int_{0}^{1}\int_0^T\prod_{j=1}^k|\zeta(\tfrac{1}{2}+i(t+h_j))|^{2\beta}dt dh_1\cdots dh_k.
\end{equation}

We now apply a conjecture of Conrey et al.~\cite{cfkrs2} concerning shifted moments of the zeta function\footnote{The conjecture in~\cite{cfkrs2} is more general in that it allows for different shifts, as well as integrating against different weight functions.} such as those appearing in the inner integral in \eqref{eq:mom_expanded}.  Specialising this conjecture to our situation (and to our notation) we have the following.

\begin{conjecture}[Conrey et al.~\cite{cfkrs2}]
  Take $k, \beta\in\mathbb{N}$, and $\underline{h}=(h_1,\dots,h_k)$ with $h_j\in\mathbb{R}$.  Then
  \begin{equation}\label{conj:cfkrs}
    \frac{1}{T}\int_{0}^T\prod_{j=1}^k|\zeta(\tfrac{1}{2}+i(t+h_j))|^{2\beta}dt=\frac{1}{T}\int_0^TP_{k,\beta}(\log\tfrac{t}{2\pi},\underline{h})dt+O(T^{-\delta}),
  \end{equation}
  for some $\delta>0$, where 
  \begin{equation}\label{def:p}
    P_{k,\beta}(x;\underline{h})\coloneqq\frac{(-1)^{k\beta}}{(k\beta)!^2}\frac{1}{(2\pi i)^{2k\beta}}\oint\cdots\oint\frac{G(z_1,\dots,z_{2k\beta})\Delta(z_1,\dots,z_{2k\beta})^2}{e^{\frac{x}{2}\sum_{j=1}^{k\beta}z_{k\beta+j}-z_j}\prod_{j=1}^{2k\beta}\prod_{l=1}^{k}(z_j-ih_l)^{2\beta}}dz_1\cdots dz_{2k\beta}
  \end{equation}
  in which $\Delta(x_1,\dots,x_n)=\prod_{i<j}(x_j-x_i)$ is the Vandermonde determinant, and the contours are small circles surrounding the poles.  Additionally,
  \begin{equation}\label{def:g_func}
    G(z_1,\dots,z_{2k\beta})\coloneqq A_{k\beta}(z_1,\dots,z_{2k\beta})\prod_{1\leq i\leq k\beta < j\leq 2k\beta}\zeta(1+z_i-z_j)
  \end{equation}
  which includes the Euler product\footnote{Throughout we write $e(\theta)=\exp(2\pi i\theta)$.}
  \begin{equation}\label{def_A}
    A_{k\beta}(\underline{z})\coloneqq\prod_p \prod_{1\leq l\leq k\beta < m\leq 2k\beta} (1-p^{z_m-z_l-1})\int_0^1\prod_{j=1}^{k\beta}\left(1-\frac{e(\theta)}{p^{\frac{1}{2}+z_j}}\right)^{-1}\left(1-\frac{e(-\theta)}{p^{\frac{1}{2}-z_{k\beta+j}}}\right)^{-1}d\theta. 
  \end{equation}
\end{conjecture}

In light of this, we define 
\begin{equation}\label{mom_p_def}
  \mom_{P_{k,\beta}}(T)\coloneqq  \frac{1}{T}\int_{0}^{1}\cdots\int_{0}^{1}\int_0^TP_{k,\beta}(\log\tfrac{t}{2\pi},\underline{h})dtdh_1\cdots dh_k.
\end{equation}
Under the conjecture of Conrey et al., $\mom_{P_{k,\beta}}(T)$ should approximate $\mom_{\zeta_T}(k,\beta)$ up to a power saving in $T$.  In \eqref{mom_p_def}, the length of the integral over $h_1,\dots, h_k$ has been chosen to be $1$ in order to be explicit.  It is readily apparent from the proof that the techniques extend to any $O(1)$ interval (with respect to $T\rightarrow\infty$). 

Our main result is the following.
\begin{theorem}\label{thm:main}
  For $k,\beta\in\mathbb{N}$, one has that as $T\rightarrow\infty$
  \begin{equation}
    \mom_{P_{k,\beta}}(T)=\alpha_{k,\beta}\gamma_{k,\beta}\left(\log\tfrac{T}{2\pi}\right)^{k^2\beta^2-k+1}(1+O\left(\log^{-1}T\right)),
  \end{equation}
  where the term $\alpha_{k,\beta}=A_{k\beta}(0,\dots,0)$ contains the arithmetic information and $A_{k\beta}$ is as defined in \eqref{def_A}, and $\gamma_{k,\beta}$ denotes the remaining, non-arithmetic, contribution (further details can be found within the proof, see in particular \eqref{eq:gamma_coeff} and \eqref{pre_incomplete_gamma}).
\end{theorem}

Therefore $ \mom_{P_{k,\beta}}(T)$, which, according to the conjecture for the shifted moments of~\cite{cfkrs2} is a close approximation to $\mom_{\zeta_T}(k,\beta)$, does satisfy the same asymptotic formula as we conjecture holds for the moments of moments of the zeta function.

%%%%%%%%%%%%%%%%%%%%%%%
%                     %
% Conj L functions    %
%                     %
%%%%%%%%%%%%%%%%%%%%%%%

\subsection{Conjectures for other families of $L$-functions}\label{sec:l-function_conj}

Following Katz and Sarnak~\cite{katsar99} (see also~\cite{confar00, keasna00b, cfkrs2}), we can also consider other families of $L$-functions.  These families are split by `symmetry type',  indicating which random matrix ensemble (unitary, symplectic, or orthogonal) should be used as comparison. 

Here we give an example from each family of the relevant `moments of moments', and the corresponding conjecture for their growth.  For each symmetry type, the moments of moments comprise an average through the family and an average over small shifts near a symmetry point. The conjectures for the symplectic and orthogonal families are derived from the relevant random matrix calculations of Assiotis et al.~\cite{abk19}, which we briefly recap here.  

Let $\mathscr{G}(N)\in\{ U(N), \operatorname{Sp}(2N), \operatorname{SO}(2N)\}$, where $\operatorname{Sp}(2N)$ is the group of $2N\times 2N$ unitary symplectic matrices, and $\operatorname{SO}(2N)$ the group of $2N\times 2N$ orthogonal matrices with unit determinant.  A generalisation of the definition given by \eqref{unitary_mom} is the following. 
\begin{equation}
  \mom_{\mathscr{G}(N)}(k,\beta)\coloneqq\int_{\mathscr{G}(N)}\left(\frac{1}{2\pi}\int_0^{2\pi}|P_N(A,\theta)|^{2\beta}d\theta\right)^kdA,
\end{equation}
where the external average is over the relevant Haar measure for the chosen matrix group. 

Assiotis et al.~\cite{abk19} computed integer moments of moments for $\mathscr{G}(N)=\operatorname{Sp}(2N)$ and $\mathscr{G}(N)=\operatorname{SO}(2N)$, and their results are summarised here.

\begin{thm}[Assiotis et al.~\cite{abk19}]\label{MainTheoremSymplectic}
  Let $\mathscr{G}(N)=\operatorname{Sp}(2N)$. Let $k,\beta \in \mathbb{N}$. Then, 
  \begin{align}\label{eq:mom_sympl}
    \mom_{\operatorname{Sp}(2N)}\left(k,\beta\right)=\mathfrak{c}_{k,\beta}N^{k\beta(2k\beta+1)-k}\left(1+O_{k,\beta}\left(\tfrac{1}{N}\right)\right),
  \end{align}
  \sloppy where the leading order term coefficient $\mathfrak{c}_{k,\beta}$ is the volume of a certain convex region\footnote{Differing from the coefficient appearing in~\eqref{eq:baikea}.}. Moreover, $\mom_{\operatorname{Sp}(2N)}\left(k,\beta\right)$ is a polynomial function in $N$.
\end{thm}

\begin{thm}[Assiotis et al.~\cite{abk19}]\label{MainTheoremOrthogonal}
  Let $\mathscr{G}(N)=\operatorname{SO}(2N)$.  Let $k, \beta\in\mathbb{N}$.  Then,  
  \begin{align}
    \mom_{\operatorname{SO}(2N)}(1,1)&=2(N+1)\\
    \intertext{otherwise,}
    \mom_{\operatorname{SO}(2N)}(k,\beta)&=\tilde{\mathfrak{c}}_{k,\beta}N^{k\beta(2k\beta-1)-k}\left(1+O_{k,\beta}\left(\tfrac{1}{N}\right)\right),\label{eq:mom_orthog}
  \end{align}
  where the leading order term coefficient $\tilde{\mathfrak{c}}_{k,\beta}$ is given as a sum of volumes of certain convex regions\footnote{Again, differing from $c_{k,\beta}$ in \eqref{eq:baikea}, as well as $\mathfrak{c}_{k,\beta}$ given in \eqref{eq:mom_sympl}.}. Moreover, $\mom_{\operatorname{SO}(2N)}(k,\beta)$ is a polynomial function in $N$.  
\end{thm}
Note that these results differ from the unitary case, where the degree of the polynomial is $k^2\beta^2-k+1$,  see~\eqref{eq:baikea}.   

Returning to the families of $L$-functions, the archetypal example of a unitary $L$-function is the Riemann zeta function, as discussed in section~\ref{sec:mom}. We rewrite the family here for context, 
\begin{equation}\label{ex:unitary_family}
  \{\zeta(\tfrac{1}{2}+it)\; |\; t\geq 0\}.
\end{equation}
To each family there is an associated `height' by which the elements are ordered.  For the family given by \eqref{ex:unitary_family}, one orders by the height up the critical line, $t$. The moments of moments are defined by \eqref{def:mom_zeta}, and their conjectured asymptotic is given by \eqref{conj:main}. 

Quadratic Dirichlet $L$-functions provide an example of a symplectic family: 
\begin{equation}\label{ex:symplectic_family}
  \{L(s, \chi_d)\; |\;\ d\text{ a fundamental discriminant, } \chi_d(n)=\left(\tfrac{d}{n}\right)\},
\end{equation}
where the family is ordered by $|d|$. For $\RE(s)>1$
\begin{equation}
  L(s, \chi_d)\coloneqq\sum_{n=1}^\infty \frac{\chi_d(n)}{n^s},
\end{equation} 
and otherwise $L(s,\chi_d)$ is defined by its analytic continuation.  The moments of moments are defined as 
\begin{equation}\label{mom_symplectic}
  \mom_{L_{\chi_d}}(k,\beta)\coloneqq\frac{1}{D^*}\sideset{}{^*}\sum_{|d|\leq D}\left(\int_{0}^{1}L(\tfrac{1}{2}+it,\chi_d)^{2\beta} dt\right)^k,
\end{equation}
where the sum is only over fundamental discriminants $d$, and $D^*$ denotes the number of terms in the sum. We hence conjecture that for $k, \beta\in \mathbb{N}$, as $D\rightarrow\infty$,
\begin{equation}
  \mom_{L_{\chi_d}}(k,\beta)=\eta_{k,\beta}\mathfrak{c}_{k,\beta}\left(\log D\right)^{k\beta(2k\beta+1)-k}\left(1+O_{k,\beta}\left(\log^{-1}D\right)\right),
\end{equation}
where $\mathfrak{c}_{k,\beta}$ corresponds to the leading order coefficient in \eqref{eq:mom_sympl}, and $\eta_{k,\beta}$ contains the arithmetic information (see for example~\cite{cfkrs2}). 

For an orthogonal example we take the family of twisted elliptic curve $L$-functions, 
\begin{equation}\label{ex:orthogonal_family}
  \{L_E(s, \chi_d)\; |\ d\text{ a fundamental discriminant, } \chi_d(n)=\left(\tfrac{d}{n}\right)\},
\end{equation}
which is ordered by $|d|$ and where the $L$-function is defined by
\begin{equation}
  L_E(s,\chi_d)\coloneqq\sum_{n=1}^\infty \frac{a_n \chi_d(n)}{n^s}
\end{equation}
for $\RE(s)>1$ and by analytic continuation otherwise. The coefficients $a_n$ are the weightings from the $L$-function associated to $E$, normalised so that the critical line is shifted to $\RE(s)=1/2$. The moments of moments are then defined as 
\begin{equation}\label{mom_orthogonal}
  \mom_{L_{E}}(k,\beta)\coloneqq\frac{1}{D^*}\sideset{}{^*}\sum_{|d|\leq D}\left(\int_{0}^{1}L_E(\tfrac{1}{2}+it,\chi_d)^{2\beta} dt\right)^k,
\end{equation}
where again the sum is only over fundamental discriminants $d$, and $D^*$ denotes the number of terms in the sum. We hence conjecture that for $k, \beta\in \mathbb{N}$ and $k\neq 1\neq \beta$, as $D\rightarrow\infty$,
\begin{equation}
  \mom_{L_{E}}(k,\beta)=\xi_{k,\beta}\tilde{\mathfrak{c}}_{k,\beta}\left(\log D\right)^{k\beta(2k\beta-1)-k}\left(1+O_{k,\beta}\left(\log^{-1}D\right)\right),
\end{equation}
where $\tilde{\mathfrak{c}}_{k,\beta}$ corresponds to the leading order coefficient in \eqref{eq:mom_orthog}, and $\xi_{k,\beta}$ contains the arithmetic information (again, see for example~\cite{cfkrs2}). 

We note that for both the symplectic and the orthogonal families defined above one could also prove theorems akin to theorem~\ref{thm:main} using the techniques described in sections~\ref{sec:structure} and~\ref{sec:proofs}.  This is due to the fact that \cite{cfkrs2} also provides conjectures for the shifted moments of the $L$-functions in question.  However, such calculations follow near verbatim those for the Riemann zeta function, so we do not report the details.

%%%%%%%%%%%%%%%%%%%%%%%
%                     %
% Structure of Proof  %
%                     %
%%%%%%%%%%%%%%%%%%%%%%%

\section{Structure of proof of theorem~\ref{thm:main}}\label{sec:structure}

Within this section, we outline the proof of theorem~\ref{thm:main}.  We require a series of lemmas whose proofs are presented in section~\ref{sec:proofs}. 

To contextualise the proof, we note that the general expression for shifted moments of the zeta function conjectured by Conrey et al.~\cite{cfkrs2} (cf. \eqref{conj:cfkrs}) has an exact correspondence to characteristic polynomials in the random matrix case~\cite{cfkrs1} (quoted as theorem~$1.5.2.$ in~\cite{cfkrs2}). A generalised version of the random matrix theorem was used by the present authors when proving \eqref{eq:baikea}.  Hence, the approach outlined in this section follows similar lines to the proof of theorem~$1.2$ in~\cite{baikea19}.  However, in several places we will need different steps to account for the arithmetic factors involved.  

We first recall the definition of $P_{k,\beta}(x;\underline{h})$ from \eqref{def:p},
\begin{equation}
    P_{k,\beta}(x;\underline{h})=\frac{(-1)^{k\beta}}{(k\beta)!^2}\frac{1}{(2\pi i)^{2k\beta}}\oint\cdots\oint\frac{G(z_1,\dots,z_{2k\beta})\Delta(z_1,\dots,z_{2k\beta})^2}{e^{\frac{x}{2}\sum_{j=1}^{k\beta}z_{k\beta+j}-z_j}\prod_{j=1}^{2k\beta}\prod_{l=1}^{k}(z_j-ih_l)^{2\beta}}dz_1\cdots dz_{2k\beta}
\end{equation}
where one integrates over small circles surrounding the poles and
\begin{equation}
  G(z_1,\dots,z_{2k\beta})= A_{k\beta}(z_1,\dots,z_{2k\beta})\prod_{1\leq i\leq k\beta < j\leq 2k\beta}\zeta(1+z_i-z_j)
\end{equation}
which includes the Euler product 
\begin{equation}
  A_{k\beta}(\underline{z})=\prod_p \prod_{1\leq l\leq k\beta < m\leq 2k\beta} (1-p^{z_m-z_l-1})\int_0^1\prod_{j=1}^{k\beta}\left(1-\frac{e(\theta)}{p^{\frac{1}{2}+z_j}}\right)^{-1}\left(1-\frac{e(-\theta)}{p^{\frac{1}{2}-z_{k\beta+j}}}\right)^{-1}d\theta.
\end{equation}

Since it will be useful, we also here record an alternative form for $A_{k\beta}(\cdot)$ also due to Conrey et al.~\cite{cfkrs2} (a special case of corollary 2.6.2 in their paper, found there at (2.6.15)),
\begin{equation}\label{alt_a}
  A_{k\beta}(z_1,\dots,z_{2k\beta})=\prod_p\sum_{m=1}^{k\beta}\prod_{n\neq m}\frac{\prod_{j=1}^{k\beta}\left(1-\frac{1}{p^{1+z_j-z_{k\beta+n}}}\right)}{1-p^{z_{k\beta+n}-z_{k\beta+m}}},
\end{equation}
where additionally they show that each local factor is actually a polynomial in $p^{-1}, p^{-z_j}$, and $p^{z_{k\beta+j}}$ for $j=1,\dots,k\beta$. 

We begin by decomposing the multiple contour integral $P_{k,\beta}(x,\underline{h})$ in to a sum of multiple contour integrals, and apply a variable change so that each contour shifts to small circles surrounding the origin. Eventually (see \eqref{conj:cfkrs}--\eqref{def:p}), we will want to set $x=\log(t/2\pi)$. 

\begin{lemma}\label{integral_decomposition}
  Asymptotically for large $x$, 
  \begin{align}
    P_{k,\beta}(x;\underline{h})&\sim\sum_{l_1,\dots,l_{k-1}=0}^{2\beta}\frac{c_{\underline{l}}(k,\beta)}{(k\beta)!^2(2\pi i)^{2k\beta}}\left(\frac{x}{2}\right)^{|\cS_{k,\beta;\underline{l}}|}A_{k\beta}(i\mu_1,\dots,i\mu_{2k\beta})e^{\frac{ix}{2}\sum_{j=1}^{k\beta}\mu_j-\mu_{k\beta+j}}\nonumber\\
    &\qquad\quad\times\int_{\Gamma_0}\cdots\int_{\Gamma_0}\prod_{\substack{ m\leq k\beta<n\\\mu_n\neq \mu_m}}\zeta\left(1+\tfrac{2(v_m-v_n)}{x}+i(\mu_m-\mu_n)\right)f(\underline{v};\underline{l})\prod_{m=1}^{2k\beta}dv_m,
  \end{align}
  where $c_{\underline{l}}(k,\beta)$ is a product of binomial coefficients determined within the proof and
  \begin{equation}\label{def:f_func}
    f(\underline{v};\underline{l})\coloneqq\frac{e^{\sum_{j=1}^{k\beta}v_j-v_{k\beta+j}}\prod_{\substack{m<n\\\mu_m=\mu_n}}\left(v_n-v_m\right)^2}{\prod_{\substack{m\leq k\beta<n\\\mu_n= \mu_m}}(v_m-v_n)\prod_{m=1}^{2k\beta}v_m^{2\beta}}.
  \end{equation}
  The contours denoted by $\Gamma_0$ are small circles around the origin, and the $\mu_i$ variables are related to the $h_j$ variables - for the exact definition see \eqref{mu_vector}. Further, the following sets are useful to define, and the former appears in the statement of the lemma,
  \begin{align}
    \cS_{k,\beta;\underline{l}}&\coloneqq\{(m,n):1\leq m\leq k\beta<n\leq 2k\beta, \mu_m=\mu_n\}\label{S_set},\\
    \cT_{k,\beta;\underline{l}}&\coloneqq\{(m,n):1\leq m\leq k\beta<n\leq 2k\beta, \mu_m\neq\mu_n\}.\label{T_set}
  \end{align}
\end{lemma}

The next lemma relies on lemma~\ref{integral_decomposition} and determines the remaining $x\equiv \log(t/2\pi)$ dependence in the integrand of $P_{k,\beta}(x;\underline{h})$ after calculating the contribution from the integration over the $h_1,\dots,h_k$.  

\begin{lemma}\label{correct_size}
  As $T\rightarrow\infty$, 
  \begin{equation}
    \mom_{P_{k,\beta}}(T)\sim \alpha_{k,\beta}\gamma_{k,\beta} \left(\log\tfrac{T}{2\pi}\right)^{k^2\beta^2-k+1},
  \end{equation}
  where $\alpha_{k,\beta}= A_{k\beta}(0,\dots,0)$,
  \begin{align}
    \gamma_{k,\beta}&=\sum_{l_1,\dots,l_{k-1}=0}^{2\beta}\frac{c_{\underline{l}}(k,\beta)}{(k\beta)!^2(2\pi i)^{2k\beta}}\int_{\Gamma_0}\cdots\int_{\Gamma_0}f(\underline{v};\underline{l})\Psi_{k,\beta}(\underline{v};\underline{l})\prod_{m=1}^{2k\beta}dv_m, \label{eq:gamma_coeff}
  \end{align}
  $c_{\underline{l}}(k,\beta)$ and $f(\underline{v};\underline{l})$ are given by \eqref{binom_prod} and \eqref{def:f_func}, and $\Psi_{k,\beta}(\underline{v};\underline{l})$ is a multiple integral defined within the proof, see~\eqref{pre_incomplete_gamma}.
\end{lemma}

The final lemma establishes non-vanishing of the leading order coefficient.

\begin{lemma}\label{positivity_coefficient}
  For $k, \beta\in\mathbb{N}$, the coefficient $\alpha_{k,\beta}\gamma_{k,\beta}$ is non-zero. 
\end{lemma}

By combining the statements of lemma~\ref{correct_size} and lemma~\ref{positivity_coefficient}, we arrive at the statement of theorem~\ref{thm:main}

\begin{equation}
  \mom_{P_{k,\beta}}(T)=\alpha_{k,\beta}\gamma_{k,\beta}\left(\log\tfrac{T}{2\pi}\right)^{k^2\beta^2-k+1}\left(1+O\left(\log^{-1}\tfrac{T}{2\pi}\right)\right),
\end{equation}
for $k,\beta\in\mathbb{N}$. 

%%%%%%%%%%%%%%%%%%%
%                                                     %
%                                                     %
%                      Proofs                     %
%                                                     %
%                                                     %
%%%%%%%%%%%%%%%%%%%

\section{Proofs of lemmas~\ref{integral_decomposition}--\ref{positivity_coefficient}}\label{sec:proofs}

In this section we give the proofs of lemma~\ref{integral_decomposition}, lemma~\ref{correct_size}, and lemma~\ref{positivity_coefficient} from section~\ref{sec:structure}.

%%%%%%%%%%%%%%%%%%%
%      Decomp Integral Lemma         %
%%%%%%%%%%%%%%%%%%%
\begin{proof}[Proof of lemma~\ref{integral_decomposition}]

We recall the statement of lemma~\ref{integral_decomposition}. 

\begin{lemmum}
   Asymptotically for large $x$, 
  \begin{align}
    P_{k,\beta}(x;\underline{h})&\sim\sum_{l_1,\dots,l_{k-1}=0}^{2\beta}\frac{c_{\underline{l}}(k,\beta)}{(k\beta)!^2(2\pi i)^{2k\beta}}\left(\frac{x}{2}\right)^{|\cS_{k,\beta;\underline{l}}|}A_{k\beta}(i\mu_1,\dots,i\mu_{2k\beta})e^{\frac{ix}{2}\sum_{j=1}^{k\beta}\mu_j-\mu_{k\beta+j}}\nonumber\\
    &\qquad\quad\times\int_{\Gamma_0}\cdots\int_{\Gamma_0}\prod_{\substack{ m\leq k\beta<n\\\mu_n\neq \mu_m}}\zeta\left(1+\tfrac{2(v_m-v_n)}{x}+i(\mu_m-\mu_n)\right)f(\underline{v};\underline{l})\prod_{m=1}^{2k\beta}dv_m,
  \end{align}
  where $c_{\underline{l}}(k,\beta)$ is a product of binomial coefficients determined in the proof and
  \begin{equation}
    f(\underline{v};\underline{l})=\frac{e^{\sum_{j=1}^{k\beta}v_j-v_{k\beta+j}}\prod_{\substack{m<n\\\mu_m=\mu_n}}\left(v_n-v_m\right)^2}{\prod_{\substack{m\leq k\beta<n\\\mu_n= \mu_m}}(v_m-v_n)\prod_{m=1}^{2k\beta}v_m^{2\beta}}.
  \end{equation}
\end{lemmum}

Recall the definition of $P_{k,\beta}(x,\underline{h})$ from \eqref{def:p},
\begin{equation}\label{def:p_proof}
  P_{k,\beta}(x;\underline{h})=\frac{(-1)^{k\beta}}{(k\beta)!^2}\frac{1}{(2\pi i)^{2k\beta}}\oint\cdots\oint\frac{G(z_1,\dots,z_{2k\beta})\Delta(z_1,\dots,z_{2k\beta})^2}{e^{\frac{x}{2}\sum_{j=1}^{k\beta}z_{k\beta+j}-z_j}\prod_{j=1}^{2k\beta}\prod_{l=1}^{k}(z_j-ih_l)^{2\beta}}dz_1\cdots dz_{2k\beta}
\end{equation}
where the contours are small circles around the poles, and $G(\underline)$ is given by \eqref{def:g_func}.

We start by decomposing the $2k\beta$ contours in \eqref{def:p_proof} so that each contour consists of $k$ small circles encompassing the poles at $ih_1,\dots,ih_k$ with connecting straight lines whose contributions cancel out.  Denote a small circular contour encircling $ih_j$ by $\Gamma_{i h_j}$.  Then,
\begin{align}
  P_{k,\beta}(x,\underline{h})&=\frac{(-1)^{k\beta}}{(k\beta)!^2}\frac{1}{(2\pi i)^{2k\beta}}\sum_{\eps_j\in\{1,\dots,k\}}Q_{k,\beta}(x,\underline{h};\eps_1,\dots,\eps_{2k\beta}),\label{p_decomp}
\end{align}
where 
\begin{equation}\label{eq:qMCI}
  Q_{k,\beta}(x,\underline{h};\eps_1,\dots,\eps_{2k\beta})=\int_{\Gamma_{ih_{\eps_1}}}\cdots\int_{\Gamma_{ih_{\eps_{2k\beta}}}}\frac{G(z_1,\dots,z_{2k\beta})\Delta(z_1,\dots,z_{2k\beta})^2}{e^{\frac{x}{2}\sum_{j=1}^{k\beta}z_{k\beta+j}-z_j}\prod_{j=1}^{2k\beta}\prod_{l=1}^{k}(z_j-ih_l)^{2\beta}}dz_1\cdots dz_{2k\beta}
\end{equation}
is the multiple contour integral $P_{k,\beta}(x,\underline{h})$ with the $2k\beta$ contours each specialised around one of the $k$ poles determined by the vector $\underline{\eps}=(\eps_1,\dots,\eps_{2k\beta})$.

Due to the symmetric nature of the integrand, many of the summands evaluate to zero. Such a fact was established in lemma~$3.2$ in~\cite{baikea19}, and is paraphrased here. 
\begin{lemmum}
  Write $\underline{\eps}=(\eps_1,\dots,\eps_{2k\beta})$ for a choice of contours appearing as a summand in \eqref{p_decomp}.  Then the only choices of $\underline{\eps}$ leading to a non-zero summand are those where each pole is equally represented in $\underline{\eps}$.  That is: if $n_i$ counts the number of occurrences of $i$ in $\underline{\eps}$, for $i\in\{1,\dots,k\}$, then the summand corresponding to the choice $\underline{\eps}$ is identically zero unless $n_1=\cdots=n_{k}=2\beta$.
\end{lemmum}

\begin{remark}
  The above lemma was formulated in~\cite{baikea19} for a slightly different integral, namely
  \begin{equation}\label{eq:oldMCI}
    \int_{\Gamma_{ih_{\eps_1}}}\cdots\int_{\Gamma_{ih{\eps_{2k\beta}}}}\frac{e^{-N(z_{k\beta+1}+\cdots+z_{2k\beta})}\Delta(z_1,\dots,z_{2k\beta})^2dz_1\cdots dz_{2k\beta}}{\prod_{m\leq k\beta<n}\left(1-e^{z_n-z_m}\right)\prod_{m=1}^{2k\beta}\prod_{n=1}^{k}(z_m-ih_n)^{2\beta}}.
  \end{equation}
  However, the structure of \eqref{eq:oldMCI} only differs trivially to that of \eqref{eq:qMCI}.  The key difference is between the terms of the form $(1-\exp({z_n-z_m}))^{-1}$ in \eqref{eq:oldMCI} versus $\zeta(1+z_m-z_n)$ in \eqref{eq:qMCI}.  However, they share the same analytic structure since both have a simple pole at $z_n=z_m$. Therefore the proof presented in~\cite{baikea19} holds as well for~\eqref{eq:qMCI}. 
\end{remark}

Hence \eqref{p_decomp} becomes
\begin{equation}
  P_{k,\beta}(x,\underline{h})=\frac{(-1)^{k\beta}}{(k\beta)!^2}\frac{1}{(2\pi i)^{2k\beta}}\sum_{l_1=0}^{2\beta}\cdots\sum_{l_{k-1}=0}^{2\beta}c_{\underline{l}}(k,\beta)Q_{k,\beta}(x,\underline{h};\underline{l}),\label{p_decomp_specialised}
\end{equation}
where $\underline{l}=(l_1,\dots,l_{k-1})$ and $Q_{k,\beta}(x,\underline{h};\underline{l})$ is the integral $Q_{k,\beta}(x,\underline{h};\underline{\eps})$ with contours given by 
\[\underline{\eps}=(\overbrace{1,\dots,1}^{l_1},\overbrace{2,\dots,2}^{l_2},\dots,\overbrace{k-1,\dots,k-1}^{l_{k-1}},\overbrace{k,\dots,k}^{2\beta},\overbrace{k-1,\dots,k-1}^{2\beta-l_{k-1}},\dots,\overbrace{1,\dots,1}^{2\beta-l_1}),\]
and the coefficient $c_{\underline{l}}(k,\beta)$ is a product of binomial coefficients:
\begin{align}
  c_{\underline{l}}(k,\beta)&=\binom{k\beta}{l_1}\binom{k\beta-l_1}{l_2}\binom{k\beta-(l_1+l_2)}{l_3}\cdots\binom{k\beta-\sum_{m=1}^{k-2}l_m}{l_{k-1}}\nonumber\\
  &\quad\times\binom{k\beta}{2\beta-l_1}\binom{(k-2)\beta+l_1}{2\beta-l_2}\cdots\binom{k\beta-\sum_{m=1}^{k-2}(2\beta-l_m)}{2\beta-l_{k-1}}.\label{binom_prod}
\end{align} 

We now apply the following change of variables so to shift all the contours to be small circles around the origin
\[z_n=\frac{2v_n}{x}+i\mu_n,\]
where
\begin{equation}\label{mu_vector}
  \mu_n=
  \begin{cases}
    h_1,&\text{if }n\in\{1,\dots,l_1\}\cup\{2(k-1)\beta+1+l_1,\dots,2k\beta\}\\
    h_2,&\text{if }n\in\{l_1+1,\dots,l_1+l_2\}\cup\{2(k-2)\beta+1+l_1+l_2,\dots,2(k-1)\beta+l_1\}\\
    \vdots&\quad\vdots\\
    h_{k-1},&\text{if }n\in\{\sum_{m=1}^{k-2}l_m+1,\dots,\sum_{m=1}^{k-1}l_m\}\cup\{2\beta+1+\sum_{m=1}^{k-1}l_m,\dots,4\beta+\sum_{m=1}^{k-2}l_m\}\\
    h_k,&\text{if }n\in\{\sum_{m=1}^{k-1}l_m+1,\dots,\sum_{m=1}^{k-1}l_m+2\beta\}.\end{cases}
\end{equation}
Note that $\mu_n$ is dependent on the choice of $\underline{l}$ (i.e. the exact ordering of the contours), but this is suppressed from the notation for simplicity.

Using the fact that the $\zeta$-function has a simple pole at $1$ with residue $1$, the resulting integrand of $Q_{k,\beta}(x,\underline{h};\underline{l})$ as $x\rightarrow\infty$ is
\begin{align}
  \big(1+O\big(\tfrac{1}{x}&\big)\big)A_{k\beta}\big(\tfrac{2v_1}{x}+i\mu_1,\dots,\tfrac{2v_{2k\beta}}{x}+i\mu_{2k\beta}\big)e^{\frac{ix}{2}\sum_{j=1}^{k\beta}\mu_j-\mu_{k\beta+j}}e^{\sum_{j=1}^{k\beta}v_j-v_{k\beta+j}}\nonumber\\
  &\qquad\times\prod_{\substack{m<n\\\mu_m\neq\mu_n}}(i\mu_n-i\mu_{n})^2\prod_{\substack{1\leq m\leq k\beta<n\leq 2k\beta\\\mu_n\neq \mu_m}}\zeta\big(1+\tfrac{2(v_m-v_n)}{x}+i(\mu_m-\mu_n)\big)\nonumber\\
  &\qquad\times\frac{\prod_{\substack{m<n\\\mu_m=\mu_n}}\left(\tfrac{2(v_n-v_m)}{x}\right)^2}{\prod_{\substack{1\leq m\leq k\beta<n\leq 2k\beta\\\mu_n= \mu_m}}\left(\tfrac{2(v_m-v_n)}{x}\right)\prod_{m=1}^{2k\beta}\prod_{n=1}^{k}\left(\tfrac{2v_m}{x}-i(h_n-\mu_m)\right)^{2\beta}}\prod_{m=1}^{2k\beta}\tfrac{2dv_m}{x}\nonumber\\
  &=\left(1+O\left(\tfrac{1}{x}\right)\right)\left(\tfrac{x}{2}\right)^{-2k\beta}A_{k\beta}(i\mu_1,\dots,i\mu_{2k\beta})e^{\frac{ix}{2}\sum_{j=1}^{k\beta}\mu_j-\mu_{k\beta+j}}e^{\sum_{j=1}^{k\beta}v_j-v_{k\beta+j}}\nonumber\\
  &\qquad\times\prod_{\substack{m<n\\\mu_m\neq\mu_n}}(i\mu_m-i\mu_n)^2\prod_{\substack{m\leq k\beta<n\\\mu_n\neq \mu_m}}\zeta\big(1+\tfrac{2(v_m-v_n)}{x}+i(\mu_m-\mu_n)\big)\nonumber\\
  &\qquad\times\frac{\prod_{\substack{m<n\\\mu_m=\mu_n}}\left(\tfrac{2(v_n-v_m)}{x}\right)^2\prod_{m=1}^{2k\beta}\left(\tfrac{2v_m}{x}\right)^{-2\beta}}{\prod_{\substack{m\leq k\beta<n\\\mu_n= \mu_m}}\left(\tfrac{2(v_m-v_n)}{x}\right)\prod_{\substack{m<n\\\mu_m\neq\mu_n}}(i\mu_m-i\mu_n)^2}\prod_{m=1}^{2k\beta}dv_m\\
  &=\left(1+O\left(\tfrac{1}{x}\right)\right)\left(\tfrac{x}{2}\right)^{4k\beta^2-2k\beta}A_{k\beta}(i\mu_1,\dots,i\mu_{2k\beta})e^{\frac{ix}{2}\sum_{j=1}^{k\beta}\mu_j-\mu_{k\beta+j}}e^{\sum_{j=1}^{k\beta}v_j-v_{k\beta+j}}\nonumber\\
  &\qquad\times\prod_{\substack{ m\leq k\beta<n\\\mu_n\neq \mu_m}}\zeta\big(1+\tfrac{2(v_m-v_n)}{x}+i(\mu_m-\mu_n)\big)\frac{\prod_{\substack{m<n\\\mu_m=\mu_n}}\left(\tfrac{2(v_n-v_m)}{x}\right)^2}{\prod_{\substack{m\leq k\beta<n\\\mu_n= \mu_m}}\tfrac{2(v_m-v_n)}{x}}\prod_{m=1}^{2k\beta}\tfrac{dv_m}{v_m^{2\beta}}.\label{var_sub_final}
\end{align}

The power of $x$ coming from the terms that originated from the Vandermonde (i.e. the numerator of the fraction in \eqref{var_sub_final}) is calculated to be $2k\beta(2\beta-1)$. We now extract the $x$ dependence remaining in the denominator of the integrand~\eqref{var_sub_final}.  First, define the following sets,
\begin{align}
  \cS_{k,\beta;\underline{l}}&\coloneqq\{(m,n):1\leq m\leq k\beta<n\leq 2k\beta, \mu_m=\mu_n\}\\
  \cT_{k,\beta;\underline{l}}&\coloneqq\{(m,n):1\leq m\leq k\beta<n\leq 2k\beta, \mu_m\neq\mu_n\}.
\end{align}
Thus, $|\cS_{k,\beta;\underline{l}}|+|\cT_{k,\beta;\underline{l}}|=k^2\beta^2$, and the pairs $(m,n)\in \cS_{k,\beta;\underline{l}}$ are precisely those involved in the denominator of \eqref{var_sub_final}. As it will be useful, we also briefly record that 
\[(-1)^{|\cS_{k,\beta;\underline{l}}|}=(-1)^{\sum_{j=1}^kl_j(2\beta-l_j)}=(-1)^{k\beta}.\]
Hence the integrand of $Q_{k,\beta}(x,\underline{h};\underline{l})$ is
\begin{align}
  &\left(1+O\left(\tfrac{1}{x}\right)\right)\left(\tfrac{x}{2}\right)^{|\cS_{k,\beta;\underline{l}}|}(-1)^{k\beta}A_{k\beta}(i\mu_1,\dots,i\mu_{2k\beta})e^{\frac{ix}{2}\sum_{j=1}^{k\beta}\mu_j-\mu_{k\beta+j}}e^{\sum_{j=1}^{k\beta}v_j-v_{k\beta+j}}\nonumber\\
  &\qquad\times\prod_{(m,n)\in \cT_{k,\beta;\underline{l}}}\zeta\left(1+\tfrac{2(v_m-v_n)}{x}+i(\mu_m-\mu_n)\right)\frac{\prod_{\substack{m<n\\\mu_m=\mu_n}}\left(v_n-v_m\right)^2}{\prod_{\substack{m\leq k\beta<n\\\mu_n= \mu_m}}(v_m-v_n)}\prod_{m=1}^{2k\beta}\tfrac{dv_m}{v_m^{2\beta}}.
\end{align}
Thus, as $x\rightarrow\infty$,
\begin{align}
  P_{k,\beta}(x,\underline{h})&\sim\sum_{l_1,\dots,l_{k-1}=0}^{2\beta}\frac{c_{\underline{l}}(k,\beta)}{(k\beta)!^2(2\pi i)^{2k\beta}}\left(\tfrac{x}{2}\right)^{|\cS_{k,\beta;\underline{l}}|}A_{k\beta}(i\mu_1,\dots,i\mu_{2k\beta})e^{\frac{ix}{2}\sum_{j=1}^{k\beta}\mu_j-\mu_{k\beta+j}}\nonumber\\
  &\quad\times\int_{\Gamma_0}\cdots\int_{\Gamma_0}\prod_{(m,n)\in \cT_{k,\beta;\underline{l}}}\zeta\left(1+\tfrac{2(v_m-v_n)}{x}+i(\mu_m-\mu_n)\right)f(\underline{v};\underline{l})\prod_{m=1}^{2k\beta}dv_m,
\end{align}
where the terms in the integrand with no $h$ dependence have been collectively denoted by $f(\underline{v};\underline{l})$, so
\begin{equation}
  f(\underline{v};\underline{l})\coloneqq\frac{e^{\sum_{j=1}^{k\beta}v_j-v_{k\beta+j}}\prod_{\substack{m<n\\\mu_m=\mu_n}}\left(v_n-v_m\right)^2}{\prod_{\substack{m\leq k\beta<n\\\mu_n= \mu_m}}(v_m-v_n)\prod_{m=1}^{2k\beta}v_m^{2\beta}},\label{none_h_terms}
\end{equation}
which concludes the proof. 

\end{proof}

%%%%%%%%%%%%%%%%%%%
%          Correct size of $x$	      %
%%%%%%%%%%%%%%%%%%%

\begin{proof}[Proof of lemma~\ref{correct_size}]

  We recall the statement of lemma~\ref{correct_size}. 

  \begin{lemmum}
    \[\mom_{P_{k,\beta}}(T)\sim \alpha_{k,\beta}\gamma_{k,\beta} \left(\log\tfrac{T}{2\pi}\right)^{k^2\beta^2-k+1},\]
    where $\alpha_{k,\beta}\coloneqq A_{k\beta}(0,\dots,0)$,
    \begin{align*}
      \gamma_{k,\beta}&\coloneqq\sum_{l_1,\dots,l_{k-1}=0}^{2\beta}\frac{c_{\underline{l}}(k,\beta)}{(k\beta)!^2(2\pi i)^{2k\beta}}\int_{\Gamma_0}\cdots\int_{\Gamma_0}f(\underline{v};\underline{l})\Psi_{k,\beta}(\underline{v};\underline{l})\prod_{m=1}^{2k\beta}dv_m, 
    \end{align*}
    $c_{\underline{l}}(k,\beta)$ and $f(\underline{v};\underline{l})$ are given by \eqref{binom_prod} and \eqref{def:f_func} and $\Psi_{k,\beta}(\underline{v};\underline{l})$ is a multiple integral defined within the proof, see~\eqref{pre_incomplete_gamma}. 
  \end{lemmum}

  Set $x\coloneqq\log(t/2\pi)$. By lemma~\ref{integral_decomposition}, we have that
    \begin{align}
    \mom_{P_{k,\beta}}(T)&= \frac{1}{T}\int_{0}^{1}\cdots\int_{0}^{1}\int_0^TP_{k,\beta}(\log\tfrac{t}{2\pi},\underline{h})dtdh_1\cdots dh_k\nonumber\\
    &\sim\sum_{l_1,\dots,l_{k-1}=0}^{2\beta}\frac{c_{\underline{l}}(k,\beta)}{T(k\beta)!^2(2\pi i)^{2k\beta}}\nonumber\\
    &\qquad\times \int_{0}^{1}\cdots\int_{0}^{1}\int_0^T A_{k\beta}(i\mu_1,\dots,i\mu_{2k\beta}) e^{\frac{ix}{2}\sum_{j=1}^{k\beta}\mu_j-\mu_{k\beta+j}}\left(\tfrac{x}{2}\right)^{|\cS_{k,\beta;\underline{l}}|}\nonumber\\
    &\qquad\times\int_{\Gamma_0}\cdots\int_{\Gamma_0}f(\underline{v};\underline{l})\prod_{(m,n)\in\cT_{k,\beta;\underline{l}}}\zeta\big(1+\tfrac{2(v_m-v_n)}{x}+i(\mu_m-\mu_n)\big)\prod_{m=1}^{2k\beta}dv_mdt \prod_{n=1}^kdh_n.\label{mom_lemma1}
    \end{align}
  
  We begin by firstly rewriting the product of zeta functions in the innermost integrand of \eqref{mom_lemma1} using the definition of $\mu_1,\dots, \mu_{2k\beta}$, see \eqref{mu_vector}, 
  \begin{align}
    \prod_{\substack{ (m,n)\in\cT_{k,\beta;\underline{l}}}}&\zeta\left(1+\tfrac{2(v_m-v_n)}{x}+i(\mu_m-\mu_n)\right)\nonumber\\
    &=\prod_{1\leq\sigma<\tau\leq k}\prod_{(m,n)\in\cV^+_{\sigma,\tau}}\zeta\left(1+\tfrac{2(v_m-v_n)}{x}+i(h_\sigma-h_\tau)\right)\prod_{(m,n)\in\cV^-_{\sigma,\tau}}\zeta\left(1+\tfrac{2(v_m-v_n)}{x}-i(h_\sigma-h_\tau)\right).\label{zeta_expanded}
  \end{align}
  where the sets $\cV^+_{\sigma,\tau}$, $\cV^-_{\sigma,\tau}$ partition $\cT_{k,\beta;\underline{l}}$,
  \begin{align}
    \cV^+_{\sigma,\tau}&\coloneqq\{(m,n)\in \cT_{k,\beta;\underline{l}}:\mu_m-\mu_n=h_\sigma-h_\tau\},\label{V_subset1}\\
    \cV^-_{\sigma,\tau}&\coloneqq\{(m,n)\in \cT_{k,\beta;\underline{l}}:\mu_m-\mu_n=h_\tau-h_\sigma\}.\label{V_subset2}
  \end{align}

  Combining \eqref{mom_lemma1} and \eqref{zeta_expanded} we find
  \begin{align}
    \mom_{P_{k,\beta}}(T)&\sim\sum_{l_1,\dots,l_{k-1}=0}^{2\beta}\frac{c_{\underline{l}}(k,\beta)}{T(k\beta)!^2(2\pi i)^{2k\beta}}\nonumber\\
    &\qquad\times \int_{0}^{1}\cdots\int_{0}^{1}\int_0^T A_{k\beta}(i\mu_1,\dots,i\mu_{2k\beta}) e^{\frac{ix}{2}\sum_{j=1}^{k\beta}\mu_j-\mu_{k\beta+j}}\left(\tfrac{x}{2}\right)^{|\cS_{k,\beta;\underline{l}}|}\nonumber\\
    &\qquad\quad\times\int_{\Gamma_0}\cdots\int_{\Gamma_0}f(\underline{v};\underline{l})\prod_{1\leq\sigma<\tau\leq k}\prod_{(m,n)\in\cV^+_{\sigma,\tau}}\zeta\left(1+\tfrac{2(v_m-v_n)}{x}+i(h_\sigma-h_\tau)\right)\nonumber\\
    &\qquad\qquad\qquad\times \prod_{(m,n)\in\cV^-_{\sigma,\tau}}\zeta\left(1+\tfrac{2(v_m-v_n)}{x}-i(h_\sigma-h_\tau)\right)\prod_{m=1}^{2k\beta}dv_mdt\prod_{n=1}^kdh_n.
    \end{align}
  Next we switch the order of integration
  \begin{align}
    \mom_{P_{k,\beta}}(T)&\sim\sum_{l_1,\dots,l_{k-1}=0}^{2\beta}\frac{c_{\underline{l}}(k,\beta)}{T(k\beta)!^2(2\pi i)^{2k\beta}}\int_0^T\left(\tfrac{x}{2}\right)^{|\cS_{k,\beta;\underline{l}}|}\int_{\Gamma_0}\cdots\int_{\Gamma_0}f(\underline{v};\underline{l})\nonumber\\
    &\qquad\times\int_{0}^{1}\cdots\int_{0}^{1} A_{k\beta}(i\mu_1,\dots,i\mu_{2k\beta})\prod_{1\leq\sigma<\tau\leq k}\prod_{(m,n)\in\cV^+_{\sigma,\tau}}\zeta\left(1+\tfrac{2(v_m-v_n)}{x}+i(h_\sigma-h_\tau)\right)\nonumber\\
    &\qquad\qquad\times \prod_{(m,n)\in\cV^-_{\sigma,\tau}}\zeta\left(1+\tfrac{2(v_m-v_n)}{x}-i(h_\sigma-h_\tau)\right)e^{\frac{ix}{2}\sum_{j=1}^{k\beta}\mu_j-\mu_{k\beta+j}}\prod_{n=1}^kdh_n\prod_{m=1}^{2k\beta}dv_mdt,\label{switched_int}
  \end{align}
  and focus on the inner integrals over $h_1,\dots, h_k$. 

  We use the structure of $(\mu_1,\dots,\mu_{2k\beta})$ in order to write the exponential term in the innermost integrand of \eqref{switched_int} explicitly in terms of $h_1,\dots,h_{2k\beta}$,
  %\begin{align}
  %\sum_{j=1}^{k\beta}\mu_j-\mu_{k\beta+j}&=\sum_{j=1}^{k-1}l_jh_j+(k\beta-\sum_{j=1}^{k-1}l_j)h_k-\sum_{j=1}^{k-1}(2\beta-l_j)h_j-(\beta(2-k)+\sum_{j=1}^{k-1}l_j)h_k\\
  %&=2\sum_{j=1}^{k-1}h_j(l_j-\beta)+2h_k(\beta(k-1)-\sum_{j=1}^{k-1}l_j)\\
  %&=2\left[\sum_{j=1}^{k-1}h_j(l_j-\beta)+h_k\sum_{j=1}^{k-1}(\beta-l_j)\right]\\
  %&=2\sum_{j=1}^{k-1}(l_j-\beta)(h_j-h_k).
  %\end{align}
  \begin{equation}
    \exp\left(\frac{ix}{2}\sum_{j=1}^{k\beta}\mu_j-\mu_{k\beta+j}\right)=\exp\left(ix\sum_{j=1}^{k-1}(l_j-\beta)(h_j-h_k)\right).
  \end{equation}

  Then, we use the Laurent expansion for the zeta function around its pole, 
  \begin{equation}
    \zeta(1+s)=\frac{1}{s}+\sum_{j=0}^{\infty}c_js^j,
  \end{equation}
  where $c_j=(-1)^j\mathfrak{s}_j/j!$ and $\mathfrak{s}_j$ are the Stieltjes constants (so $c_0$ is the Euler-Mascheroni constant). 

  In this calculation, we will be taking $s\sim 1/x$, so the main contribution to \eqref{switched_int} comes from just taking the first term in the Laurent series for each zeta function in the integrand. Hence the integral over the $h_1,\dots,h_k$ becomes

  \begin{align}
    \int_{0}^{1}&\cdots\int_{0}^{1} A_{k\beta}(i\mu_1,\dots,i\mu_{2k\beta})\prod_{1\leq\sigma<\tau\leq k}\prod_{(m,n)\in\cV^+_{\sigma,\tau}}\zeta\left(1+\tfrac{2(v_m-v_n)}{x}+i(h_\sigma-h_\tau)\right)\nonumber\\
    &\times \prod_{(m,n)\in\cV^-_{\sigma,\tau}}\zeta\left(1+\tfrac{2(v_m-v_n)}{x}-i(h_\sigma-h_\tau)\right)e^{\frac{ix}{2}\sum_{j=1}^{k\beta}\mu_j-\mu_{k\beta+j}}dh_1\cdots dh_k\nonumber\\
    &\sim\int_{0}^{1}\cdots\int_{0}^{1} \frac{A_{k\beta}(i\mu_1,\dots,i\mu_{2k\beta})e^{ix\sum_{j=1}^{k-1}(l_j-\beta)(h_j-h_k)}dh_1\cdots dh_k}{\prod_{1\leq\sigma<\tau\leq k}\prod_{(m,n)\in\cV^+_{\sigma,\tau}}\left(\tfrac{2(v_m-v_n)}{x}+i(h_\sigma-h_\tau)\right)\prod_{(m,n)\in\cV^-_{\sigma,\tau}}\left(\tfrac{2(v_m-v_n)}{x}-i(h_\sigma-h_\tau)\right)}.\label{h_int_zeta_exp}
  \end{align}
  Observe that each term in the integrand of \eqref{h_int_zeta_exp} is a function of the differences $\delta_j\coloneqq h_j-h_k$, $j=1,\dots,k$.  For the exponential term this is immediate.  For $A_{k\beta}(\cdot)$ we use \eqref{alt_a}, writing $A_{k\beta;p}(\cdot)$ for each local factor, 
  \begin{align}
    A_{k\beta;p}(i\mu_1,\dots,i\mu_{2k\beta})=\sum_{m=1}^{k\beta}\prod_{n\neq m}\frac{\prod_{j=1}^{k\beta}\left(1-\frac{1}{p^{1+i(\mu_j-h_k)}p^{-i(\mu_{k\beta+n}-h_k)}}\right)}{1-p^{i(\mu_{k\beta+n}-h_k)}p^{-i(\mu_{k\beta+m}-h_k)}}.
  \end{align}
  Similarly for the terms in the denominator, we write $h_\sigma-h_\tau=(h_\sigma-h_k)-(h_\tau-h_k)$.  We now perform a change of variables $2\delta_j=x(h_j-h_k)$ and use the fact that $A_{k\beta}$ is analytic in a neighbourhood of $(0,\dots,0)$.  Hence
  \begin{align}
    \int_{0}^{1}&\cdots\int_{0}^{1}\frac{A_{k\beta}(i\mu_1,\dots,i\mu_{2k\beta})e^{ix\sum_{j=1}^{k-1}(l_j-\beta)(h_j-h_k)}dh_1\cdots dh_k}{\prod_{1\leq\sigma<\tau\leq k}\prod_{(m,n)\in\cV^+_{\sigma,\tau}}\left(\tfrac{2(v_m-v_n)}{x}+i(h_\sigma-h_\tau)\right)\prod_{(m,n)\in\cV^-_{\sigma,\tau}}\left(\tfrac{2(v_m-v_n)}{x}-i(h_\sigma-h_\tau)\right)}\nonumber\\
    &\sim \int_{0}^{\frac{x}{2}}\cdots\int_{0}^{\frac{x}{2}} \frac{\left(\frac{x}{2}\right)^{|\cT_{k,\beta;\underline{l}}|-{k-1}}A_{k\beta}\left(0,\dots,0\right)e^{2i\sum_{j=1}^{k-1}(l_j-\beta)\delta_j}d\delta_1\cdots d\delta_{k-1}}{\prod_{1\leq\sigma<\tau\leq k}\prod_{(m,n)\in\cV^+_{\sigma,\tau}}\left(v_m-v_n+i(\delta_\sigma-\delta_\tau)\right)\prod_{(m,n)\in\cV^-_{\sigma,\tau}}\left(v_m-v_n-i(\delta_\sigma-\delta_\tau)\right)}\\
    &\sim\left(\frac{x}{2}\right)^{|\cT_{k,\beta;\underline{l}}|-{k-1}}A_{k\beta}\left(0,\dots,0\right)\Psi_{k,\beta}(\underline{v};\underline{l}),
  \end{align}
  where
  \begin{align}
    \Psi_{k,\beta}(\underline{v};\underline{l})\coloneqq\int_{0}^{\infty}\cdots&\int_{0}^{\infty} \frac{e^{2i\sum_{j=1}^{k-1}(l_j-\beta)\delta_j}}{\prod_{1\leq\sigma<\tau\leq k}\prod_{(m,n)\in\cV^+_{\sigma,\tau}}\left(v_m-v_n+i(\delta_\sigma-\delta_\tau)\right)}\nonumber\\
    &\qquad\qquad\times\frac{d\delta_1\cdots d\delta_{k-1}}{\prod_{1\leq\sigma<\tau\leq k}\prod_{(m,n)\in\cV^-_{\sigma,\tau}}\left(v_m-v_n-i(\delta_\sigma-\delta_\tau)\right)}. \label{pre_incomplete_gamma}
  \end{align}

  Then $\Psi_{k,\beta}(\underline{v};\underline{l})$ defined by \eqref{pre_incomplete_gamma} no longer depends on $x$.  We will further analyse its contribution to $\mom_{P_{k,\beta}}(T)$ in the proof of lemma~\ref{positivity_coefficient} below.

  Incorporating \eqref{pre_incomplete_gamma} in to \eqref{switched_int}, and recalling that $|\cS_{k,\beta;\underline{l}}|+|\cT_{k,\beta;\underline{l}}|=k^2\beta^2$ and $x=\log(t/2\pi)$, we find that 
  \begin{align}
    \mom_{P_{k,\beta}}(T)&\sim\sum_{l_1,\dots,l_{k-1}=0}^{2\beta}\frac{c_{\underline{l}}(k,\beta)A_{k\beta}(0,\dots,0)}{T(k\beta)!^2(2\pi i)^{2k\beta}}\int_0^T\left(\frac{x}{2}\right)^{k^2\beta^2-k+1}\int_{\Gamma_0}\cdots\int_{\Gamma_0}f(\underline{v};\underline{l})\Psi_{k,\beta}(\underline{v};\underline{l})\prod_{m=1}^{2k\beta}dv_m\; dt\\
    &=\alpha_{k,\beta}\gamma_{k,\beta}\left(\log\tfrac{T}{2\pi}\right)^{k^2\beta^2-k+1}\left(1+O\left(\log^{-1}\tfrac{T}{2\pi}\right)\right),
  \end{align}
  where  $\alpha_{k,\beta}\coloneqq A_{k\beta}(0,\dots,0)$ and 
  \begin{equation*}
    \gamma_{k,\beta}\coloneqq\sum_{l_1,\dots,l_{k-1}=0}^{2\beta}\frac{c_{\underline{l}}(k,\beta)}{(k\beta)!^2(2\pi i)^{2k\beta}}\int_{\Gamma_0}\cdots\int_{\Gamma_0}f(\underline{v};\underline{l})\Psi_{k,\beta}(\underline{v};\underline{l})\prod_{m=1}^{2k\beta}dv_m.
  \end{equation*}

\end{proof}

%%%%%%%%%%%%%%%%%%%
%             Positivity of LOC	       %
%%%%%%%%%%%%%%%%%%%

\begin{proof}[Proof of lemma~\ref{positivity_coefficient}]

  We recall the statement of lemma~\ref{positivity_coefficient}. 

  \begin{lemmum}
    The coefficient $\alpha_{k,\beta}\gamma_{k,\beta}$ appearing in the statement of lemma~\ref{correct_size} is non-zero.
  \end{lemmum}

  From lemma~\ref{correct_size}, we have that $\alpha_{k,\beta}=A_{k\beta}(0,\dots,0)$. The value of $A_n(0,\dots,0)$ was calculated by Conrey et al.~\cite{cfkrs2}. Hence, using their formula, we have

  \begin{equation}\label{arithmetic_term}
    A_{k\beta}(0,\dots,0)=\prod_p\left(1-\frac{1}{p}\right)^{(k\beta-1)^2}\sum_{m=0}^{k\beta-1}\binom{k\beta-1}{m}^2p^{-m}.
  \end{equation}

  Thus, all that remains to justify is that $\gamma_{k,\beta}\neq 0$.  To do this we appeal to the residue theorem\footnote{The initial part of the proof borrows calculations found in the proof of lemma 3.7 in~\cite{baikea19}.}.  Fix a choice of $l_1,\dots,l_{k-1}$ and notice that 
  \[\frac{c_{\underline{l}}(k,\beta)}{(k\beta)!^2}>0.\]

  To conclude, we compare to the random matrix case.  In previous work~\cite{baikea19}, recall (cf. \eqref{eq:baikea}) that the moments of the moments of unitary characteristic polynomials
  \begin{equation}
    \mom_{U(N)}(k,\beta)\coloneqq\int_{U(N)}\left(\frac{1}{2\pi}\int_0^{2\pi}|P_N(A,\theta)|^{2\beta}d\theta\right)^kdA
  \end{equation}
  were shown to satisfy
  \begin{equation}
    \mom_{U(N)}(k,\beta)=c_{k,\beta}N^{k^2\beta^2-k+1}(1+O(\tfrac{1}{N})),
  \end{equation}
  and $c_{k,\beta}$ is non-zero. The exact form\footnote{In~\cite{baikea19}, the multiple contour integral that they treat differs only from the one we consider here by a simple transformation.  For simplicity, the results presented henceforth have been translated so to be consistent with the notation within this article.} of $c_{k,\beta}$ is found by evaluating
  \begin{align*}
    c_{k,\beta}=\sum_{l_1,\dots,l_{k-1}=0}^{2\beta}\frac{c_{\underline{l}}(k,\beta)}{((k\beta)!)^2(2\pi i)^{2k\beta}}\int_{\Gamma_0}\cdots\int_{\Gamma_0}f(\underline{v};\underline{l})\Omega_{k,\beta}(\underline{v};\underline{l})\prod_{m=1}^{2k\beta}dv_m, 
  \end{align*}
  where $c_{\underline{l}}(k,\beta)$ and $f(\underline{v};\underline{l})$ are the same functions used throughout the present work (see~\eqref{binom_prod} and~\eqref{none_h_terms}),
  \begin{align}
    \Omega_{k,\beta}(\underline{v};\underline{l})=\lim_{N\rightarrow\infty}\frac{1}{(2\pi)^kN^{\cT_{k,\beta;\underline{l}}-k+1}}\int_0^{2\pi}\cdots\int_0^{2\pi}&\frac{e^{iN\sum_{j=1}^{k-1}(l_j-\beta)(\theta_j-\theta_k)}}{\prod_{(m,n)\in \cT_{k,\beta;\underline{l}}}(1-e^{\frac{v_n-v_m}{N}}e^{i(\mu_n-\mu_m)})}d\theta_1\cdots d\theta_k\\
    =\lim_{N\rightarrow\infty}\frac{1}{(2\pi)^kN^{\cT_{k,\beta;\underline{l}}-k+1}}\int_0^{2\pi}\cdots\int_0^{2\pi}&\frac{e^{iN\sum_{j=1}^{k-1}(l_j-\beta)(\theta_j-\theta_k)}}{\prod_{1\leq \sigma<\tau\leq k}\prod_{(m,n)\in \cV^+_{\sigma,\tau}}\left(1-e^{\frac{v_n-v_m}{N}}e^{i(\theta_\tau-\theta_\sigma)}\right)}\nonumber\\
    &\times \frac{d\theta_1\cdots d\theta_k}{\prod_{(m,n)\in \cV^-_{\sigma,\tau}}\left(1-e^{\frac{v_m-v_n}{N}}e^{i(\theta_\sigma-\theta_\tau)}\right)}.
  \end{align}
  The sets $\cT_{k,\beta;\underline{l}}$, $\cV^+_{\sigma,\tau}$, and $\cV^-_{\sigma,\tau}$ are as defined in \eqref{T_set}, \eqref{V_subset1}, and \eqref{V_subset2} respectively and $\mu$ matches the definition given by \eqref{mu_vector}. 

  In~\cite{baikea19}, the denominator in the integrand of $\Omega_{k,\beta}$ was expanded as a geometric series, and then the integral was explicitly calculated.  However, one could proceed in a similar way to the proof of lemma~\ref{correct_size}, i.e. noticing that the integrand is simply a function of the differences $\delta_j\coloneqq \theta_j-\theta_k$ and then use the Laurent expansion of the denominator to pull out the power of $N$. Proceeding in this way we find
  \begin{alignat}{2}
    \Omega_{k,\beta}(\underline{v};\underline{l}) &\sim \frac{1}{(2\pi)^k}\int_0^{\infty}\cdots\int_0^{\infty} &&\frac{e^{i\sum_{j=1}^{k-1}(l_j-\beta)\delta_j}}{\prod_{1\leq \sigma<\tau\leq k}\prod_{(m,n)\in \cV^+_{\sigma,\tau}}\left({v_n-v_m+i(\delta_\tau-\delta_\sigma)}\right)}\nonumber\\
    &&&\times \frac{d\delta_1\cdots d\delta_{k-1}}{\prod_{(m,n)\in \cV^-_{\sigma,\tau}}\left(v_n-v_m-i(\delta_\tau-\delta_\sigma)\right)}\\
    %&\sim \frac{1}{(2\pi)^k}\int_0^{\infty}\cdots\int_0^{\infty}\frac{e^{i\sum_{j=1}^{k-1}(l_j-\beta)\delta_j}d\delta_1\cdots d\delta_{k-1}}{\prod_{1\leq \sigma<\tau\leq k}\prod_{(m,n)\in \cV^+_{\sigma,\tau}}-\left({v_m-v_n+i(\delta_\sigma-\delta_\tau)}\right)\prod_{(m,n)\in \cV^-_{\sigma,\tau}}-\left(v_m-v_n-i(\delta_\sigma-\delta_\tau)\right)}\\
    &\sim \frac{(-1)^{|\cT_{k,\beta;\underline{l}}|}}{(2\pi)^k}\int_0^{\infty}\cdots\int_0^{\infty}&&\frac{e^{i\sum_{j=1}^{k-1}(l_j-\beta)\delta_j}}{\prod_{1\leq \sigma<\tau\leq k}\prod_{(m,n)\in \cV^+_{\sigma,\tau}}\left({v_m-v_n+i(\delta_\sigma-\delta_\tau)}\right)}\nonumber\\
    &&&\times \frac{d\delta_1\cdots d\delta_{k-1}}{\prod_{(m,n)\in \cV^-_{\sigma,\tau}}\left(v_m-v_n-i(\delta_\sigma-\delta_\tau)\right)}.
  \end{alignat}

  Thus, since $|\cS_{k,\beta;\underline{l}}|+|\cT_{k,\beta;\underline{l}}|=k^2\beta^2$, and $|\cS_{k,\beta;\underline{l}}|=\sum_{j=1}^kl_j(2\beta-l_j)$, we have 
  \[(-1)^{|\cT_{k,\beta;\underline{l}}|}=(-1)^{k\beta(k\beta-1)}=1.\] 
  So $\Omega_{k,\beta}(\underline{v};\underline{l})\sim \kappa \Psi_{k,\beta}(\underline{v};\underline{l})$ for some positive constant $\kappa$. Hence $\gamma_{k,\beta}\neq 0$ by comparison with $c_{k,\beta}$.
\end{proof}

\section{Acknowledgements}
We would like to thank Valeriya Kovaleva for helpful comments and discussions. We are grateful to the anonymous referee for their careful reading of the paper and suggestions. ECB would like to thank the Heilbronn Institute for Mathematical Research for support.   JPK is pleased to acknowledge support from ERC Advanced Grant 740900 (LogCorRM).


\begin{thebibliography}{10}
\expandafter\ifx\csname url\endcsname\relax
  \def\url#1{\texttt{#1}}\fi
\expandafter\ifx\csname urlprefix\endcsname\relax\def\urlprefix{URL }\fi
\expandafter\ifx\csname href\endcsname\relax
  \def\href#1#2{#2} \def\path#1{#1}\fi

\bibitem{cfkrs2}
J.~B. Conrey, D.~W. Farmer, J.~P. Keating, M.~O. Rubinstein, N.~C. Snaith,
  Integral moments of $l$-functions, Proceedings of the London Mathematical
  Society 91~(1) (2005) 33--104.

\bibitem{confar00}
J.~B. Conrey, D.~W. Farmer, Mean values of ${L}$-functions and symmetry,
  International Mathematics Research Notices 2000~(17) (2000) 883--908.

\bibitem{harlit18}
G.~H. Hardy, J.~E. Littlewood, Contributions to the theory of the {R}iemann
  zeta-function and the theory of the distribution of primes, Acta Mathematica
  41 (1918) 119--196.

\bibitem{ing26}
A.~E. Ingham, Mean-value theorems in the theory of the {R}iemann zeta-function,
  Proceedings of the London Mathematical Society 2~(1) (1926) 273--300.

\bibitem{congho92}
J.~B. Conrey, A.~Ghosh, Mean values of the zeta-function, {III}, in:
  Proceedings of the Amalfi Conference on Analytic Number Theory, Universit\`a
  di Salerno, 1992.

\bibitem{congon01}
J.~B. Conrey, S.~M. Gonek, High moments of the {R}iemann zeta-function, Duke
  Mathematical Journal 107~(3) (2001) 577--604.

\bibitem{ram80}
K.~Ramachandra, Some remarks on the mean value of the {R}iemann zeta-function
  and other {D}irichlet series-{II}, Hardy-Ramanujan Journal 3.

\bibitem{heabro81}
D.~R. Heath-Brown, Fractional {M}oments of the {R}iemann {Z}eta-{F}unction,
  Journal of the London Mathematical Society 2~(1) (1981) 65--78.

\bibitem{sourad13}
M.~Radziwi{\l}{\l}, K.~Soundararajan, Continuous lower bounds for moments of
  zeta and l-functions, Mathematika 59~(1) (2013) 119--128.

\bibitem{sou09}
K.~Soundararajan, Moments of the {R}iemann zeta function, Annals of Mathematics
  170 (2009) 981--993.

\bibitem{har13}
A.~J. Harper, Sharp conditional bounds for moments of the {R}iemann zeta
  function, arXiv preprint arXiv:1305.4618.

\bibitem{keasna00a}
J.~P. Keating, N.~C. Snaith, Random matrix theory and $\zeta(1/2+it)$,
  Communications in Mathematical Physics 214~(1) (2000) 57--89.

\bibitem{ck1}
B.~Conrey, J.~P. Keating, Moments of zeta and correlations of divisor-sums:
  {I}, Philosophical Transactions of the Royal Society A: Mathematical,
  Physical and Engineering Sciences 373~(2040) (2015) 20140313.

\bibitem{ck2}
B.~Conrey, J.~P. Keating, Moments of zeta and correlations of divisor-sums:
  {II}, in: Advances in the Theory of Numbers, Springer, 2015, pp. 75--85.

\bibitem{ck3}
B.~Conrey, J.~P. Keating, Moments of zeta and correlations of divisor-sums:
  {III}, Indagationes Mathematicae 26~(5) (2015) 736--747.

\bibitem{ck4}
B.~Conrey, J.~P. Keating, Moments of zeta and correlations of divisor-sums:
  {IV}, Research in Number Theory 2~(1) (2016) 24.

\bibitem{ck5}
B.~Conrey, J.~P. Keating, Moments of zeta and correlations of divisor-sums:
  {V}, Proceedings of the London Mathematical Society 118~(4) (2019) 729--752.

\bibitem{katsar99}
N.~Katz, P.~Sarnak, Zeroes of zeta functions and symmetry, Bulletin of the
  American Mathematical Society 36~(1) (1999) 1--26.

\bibitem{fyodorov12}
Y.~V. Fyodorov, G.~A. Hiary, J.~P. Keating, Freezing transition, characteristic
  polynomials of random matrices, and the {R}iemann zeta function, Physical
  Review Letters 108~(17) (2012) 170601.

\bibitem{fyodorov14}
Y.~V. Fyodorov, J.~P. Keating, Freezing transitions and extreme values: random
  matrix theory, and disordered landscapes, Philosophical Transactions of the
  Royal Society A: Mathematical, Physical and Engineering Sciences 372~(2007)
  (2014) 20120503.

\bibitem{harper1}
A.~J. Harper, The {R}iemann zeta function in short intervals [after {N}ajnudel,
  and {A}rguin, {B}elius, {B}ourgade, {R}adziwi{\l}{\l}, and {S}oundararajan],
  arXiv preprint arXiv:1904.08204.

\bibitem{harper2}
A.~J. Harper, On the partition function of the {R}iemann zeta function, and the
  {F}yodorov--{H}iary--{K}eating conjecture, arXiv preprint arXiv:1906.05783.

\bibitem{Najnudel}
J.~Najnudel, On the extreme values of the {R}iemann zeta function on random
  intervals of the critical line, Probability Theory and Related Fields
  172~(1-2) (2018) 387--452.

\bibitem{abbrs}
L.-P. Arguin, D.~Belius, P.~Bourgade, M.~Radziwi{\l}{\l}, K.~Soundararajan,
  Maximum of the {R}iemann zeta function on a short interval of the critical
  line, Communications on Pure and Applied Mathematics 72~(3) (2019) 500--535.

\bibitem{argbourad20}
L.-P. Arguin, P.~Bourgade, M.~Radziwi{\l}{\l}, The
  {F}yodorov--{H}iary--{K}eating {C}onjecture. {I.}, arXiv preprint
  arXiv:2007.00988.

\bibitem{argouirad}
L.-P. Arguin, F.~Ouimet, M.~Radziwi{\l}{\l}, Moments of the {R}iemann zeta
  function on short intervals of the critical line, arXiv preprint
  arXiv:1901.04061.

\bibitem{baikea19}
E.~C. Bailey, J.~P. Keating, On the moments of the moments of the
  characteristic polynomials of random unitary matrices, Communications in
  Mathematical Physics 371~(2) (2019) 689--726.

\bibitem{ak19}
T.~Assiotis, J.~P. Keating, Moments of moments of characteristic polynomials of
  random unitary matrices and lattice point counts, Random Matrices: Theory and
  Applications (to appear).

\bibitem{cfkrs1}
J.~B. Conrey, D.~W. Farmer, J.~P. Keating, M.~O. Rubinstein, N.~C. Snaith,
  Autocorrelation of random matrix polynomials, Communications in Mathematical
  Physics 237~(3) (2003) 365--395.

\bibitem{abk19}
T.~Assiotis, E.~C. Bailey, J.~P. Keating, On the moments of the moments of the
  characteristic polynomials of {H}aar distributed symplectic and orthogonal
  matrices, Annales de l'Institut Henri Poincar\'e D (to appear).

\bibitem{CK}
T.~Claeys, I.~Krasovsky, Toeplitz determinants with merging singularities, Duke
  Mathematical Journal 164~(15) (2015) 2897--2987.

\bibitem{fahs}
B.~Fahs, Uniform asymptotics of {T}oeplitz determinants with fisher-hartwig
  singularities, arXiv preprint arXiv:1909.07362.

\bibitem{baikea21}
E.~C. Bailey, J.~P. Keating, Moments of {M}oments and {B}ranching {R}andom {W}alks,
  Journal of Statistical Physics 182~(1) (2021) 1--24.

\bibitem{keasna00b}
J.~P. Keating, N.~C. Snaith, Random matrix theory and ${L}$-functions at
  $s=1/2$, Communications in Mathematical Physics 214~(1) (2000) 91--100.

\end{thebibliography}
\end{document}